\documentclass{amsart}
\usepackage{amsmath, amssymb, amsthm, amsfonts, amscd, mathrsfs}
\usepackage[T1]{fontenc}
\usepackage[utf8x,applemac]{inputenc}
\usepackage{graphicx,pict2e}

\usepackage{tikz}

\usepackage{hyperref}
\usepackage{amsfonts,esint}
\theoremstyle{plain}
\newtheorem{teo}{Theorem}[section]
\newtheorem{prop}[teo]{Proposition}

\newtheorem{lm}[teo]{Lemma}
\theoremstyle{definition}

\newtheorem{rem}{Remark}

\numberwithin{equation}{section}

\newcommand{\bemul}{\begin{multline*}}        

\newcommand{\pical}{\mathcal{P}}

\newcommand{\res}{\mathop{\hbox{\vrule height 7pt width .5pt depth 0pt \vrule height .5pt width 6pt depth 0pt}}\nolimits}
\newcommand{\deb}{\rightharpoonup}
\newcommand{\destar}{\overset{*}\deb}

\newcommand{\M}{\mathcal M}

\newcommand{\impl}{\Rightarrow}

\newcommand{\ve}{\varepsilon}

\newcommand{\R}{\mathbb R}

\DeclareMathOperator{\spt}{spt}
\DeclareMathOperator{\Lip}{Lip}

\begin{document}

\title[Dacorogna-Moser for minimal flows]{A Dacorogna-Moser approach to flow decomposition and minimal flow problems}
%
\author{Filippo Santambrogio}\address{Filippo Santambrogio, Laboratoire de Mathématiques d'Orsay,
 Université Paris-Sud,
 91405 Orsay cedex, FRANCE, e-mail address:\\
  {\tt filippo.santambrogio@math.u-psud.fr} ---  This paper is part of the ANR project ISOTACE ---
}
%
%
\begin{abstract} The papers describes an easy approach, based on a classical construction by Dacorogna and Moser, to prove that optimal vector fields in some minimal flow problem linked to optimal transport models (congested traffic, branched transport, Beckmann's problem\dots) are induced by a probability measure on the space of paths. This gives a new, easier, proof of a classical result by Smirnov, and allows handling optimal flows without taking care of the presence of cycles.\end{abstract}
%
%
%
\maketitle

\section*{Introduction}

Many transport optimization problems come under the form of a minimal flow problem: origins and destinations of the motion are modeled via two densities (or measures) $\mu$ and $\nu$ on a given domain $\Omega\subset\R^d$, and the unknown is how to connect them through a vector field $v:\Omega\to \R^d$ where $v(x)$ stands at every point for the density of flow passing through the point $x$, giving both the direction and the intensity of the movement. The fact that this vector field must connect $\mu$ to $\nu$ is expressed through a divergence condition: we impose $\nabla\cdot  v=\mu-\nu$, which means that the mass ``taken'' from the flow is $\mu$ and the mass ``released'' at the end of the movement is $\nu$. This description of the motion is both Eulerian and statical. Eulerian because instead of following the particles in their motion from $\mu$ to $\nu$ it looks at what happens at each point, where particles coming from different points could pass at different moments. Statical because it ignores the time dimension, and $v(x)$ represents the overall flow passing through $x$ as if we did an average over time, or if we supposed the motion to be perpetual and cyclical (at every instant some mass is poured in $\Omega$ according to the distribution $\mu$ and at the same time withdrawn according to $\nu$).

Among all the feasible vector fields $v$, we minimize a criterion, looking for the $v$ where the overall motion is minimal in a certain sense. This typically means solving a problem 
$$\min\quad\left\{ F(|v|)\;:\; \nabla\cdot v=\mu-\nu\right\},$$
where we assume that $F$ only depends (locally or not) on the intensity of $v$, i.e. on $|v|$. $F$ is obviously supposed to be increasing in $|v|$, but at least two different behaviors should be highlighted. If $F$ is convex and superlinear, then it overpenalizes high concentrations of $|v|$; on the contrary, if it is concave, subadditive and sublinear, it prefers concentration rather than dispersion. These two behaviors correspond to two different models, both meaningful in transport applications. The convex case id useful to minimize traffic congestion,  where the overall total cost that we experience is a convex function of $|v|$, for instance $\int |v|^2$. If instead we need to build a transportation network, able to carry mass from $\mu$ to $\nu$, and we need to minimize the construction and management costs of it, we will try co concentrate as much as possible the movement on some common axes. In this case the overall cost that we consider for moving a mass $m$ is of the form $g(m)$ with $g$ subadditive, i.e. $g(m_1+m_2)<g(m_1)+g(m_2)$ (which favors concentration), and typically it could be of the form $m^\alpha$, for $\alpha<1$. Notice that we changed on purpose the notation from $|v|$ to $m$, since defining such a functional $F$ is not evident in this case and usually requires $v$ to be highly concentrated. The good framework is that of vector measures (and $|v|$ denotes in this case the total variation measure of $v$). This framework also fits the convex case; in the concave case, the functional $F$ will only be finite on measures concentrated on (possibly infinite) one-dimensional rectifiable graphs, and the quantity $m$ will represent the density of $|v|$ w.r.t. the one-dimensional measure: it stands for the amount of mass (and not the density of mass) passing through a precise point, and it will only be positive on a thin set, standing for the transportation network. Notice finally that the intermediate and concentration-neutral case where we just minimize the total mass of $v$, i.e. $||v||=\int |v|$ is also interesting. It has been introduced by the spatial economist M. Beckmann in the '50s and it turns out to be equivalent to the Monge-Kantorovich problem for the optimal transport between $\mu$ and $\nu$ with cost $|x-y|$ (see \cite{Monge,Kantorovich,villani} and \cite{strang} for the equivalence . In the same paper \cite{beck} Beckmann also proposed to minimize convex costs. This idea also appears in \cite{brenierextended} as a variant of the well-known Benamou-Brenier formula in a dynamical setting, and could also be considered now as a particular case of the non-linear mobility of Dolbeaut-Nazaret-Savaré, \cite{DolNazSav} (recently, the concave case as well got a sort of Benamou-Brenier formulation, see \cite{LorenzoeGiuseppe})

For the precise expressions of the minimization problem in the concave case, called {\it branched transport} problem, we refer to \cite{xia1} and to the monograph \cite{BookIrrigation}; the model comes from classical issues in graph optimization, see \cite{Gil,GilPol}, and it translates them into a continuous framework. The convex case, with application to traffic congestion, is also classical, and comes, in a graph setting, from the works of Wardrop and Beckmann, \cite{War,BecMcGWin}. 

However, the goal of this short paper is not to enter into details of these transport models, but only to address one single question which often arises when studying them.
In many cases it is important to be able to switch from an Eulerian model to a Lagrangian one. In Lagrangian models, instead of looking at what happens at each point, one looks at what happens at each particle, describing its trajectory. Yet, since particles are considered to be indistinguishable, it is enough to consider how many particles follow each possible trajectory, i.e. giving a measure on the space of possible trajectories. Usually, we consider a probability measure $Q$, called traffic plan, on the space $C:=\Lip([0,1], \Omega)$ of Lipschitz continuous paths valued in our domain $\Omega$, and we impose that it connects the measures $\mu$ and $\nu$, by requiring  $(e_0)_\#Q=\mu$ and $(e_1)_\#Q=\nu$, where $e_t:C\to\Omega$ denotes the evaluation map at time $t$, i.e. $e_t(\omega)=\omega(t)$ for every $\omega\in C$. Then, one needs to associate a positive measure $i_Q$ to $Q$ which will represent the traffic intensity generated by $Q$ and, possibly, a vector measure $v_Q$ standing for the flow generated by the same $Q$. This has been formalized in the traffic framework in \cite{CarJimSan} (for $i_Q$) and \cite{BraCarSan} (for $v_Q$), and we will recall the precise definitions in the next section. \cite{CarJimSan} also describes the equivalence between the optimization point of view (a planner decides where every agent moves) and the equilibrium point of view (every agent chooses her own path, but her cost depends on the choices of everybody, and we look for a Nash equilibrium), for the congestion case. Even if looking for an equilibrium is not equivalent to optimizing the total cost of the agents, there is a connection between the two notions (the equilibrium actually optimizes some total cost of a similar form, which makes this a potential game). It is important to stress that equilibrium conditions can only be stated in the Lagrangian framework. 

The problem of switching from Eulerian to Lagrangian formalisms amounts at the following question: when does a vector field $v$ with $\nabla\cdot v=\mu-\nu$ is induced by a traffic plan $Q$ (i.e. whether $v=v_Q$) with $(e_0)_\#Q=\mu$ and $(e_1)_\#Q=\nu$. This is a classical question, and a general answer is provided by the well-known and classical result by Smirnov, \cite{smirnov}, which is expressed in terms of currents. In his paper, Smirnov first studies the structure of cyclical currents (i.e. with $0$ divergence, Theorems A and B of his paper), and then turns to general currents (Theorem C) by proving a decomposition of the following form: every vector field $v$ can be decomposed into two parts, one is cyclical and the other is induced by a traffic plan $Q$. Also, the mass of $v$ is equal to the sum of the two masses, which was probably the most delicate contribution of Smirnov. In particular, this implies that any $v$ without cycles is induced by a traffic plan. Thus, many papers in branched transport have tried to use this idea to make a bridge between the Eulerian model by Xia (see \cite{xia1}) and the Lagragian models by Maddalena-Solimini-Morel and then Bernot-Caselles-Morel (see \cite{MadSolMor,BerCasMor}). The main tool was first proving that optimal fields $v$ have no cycles. In this way, the possible presence of cycles has become almost an obsession for many researches that have been carried on about branched transport (the congested models have been more or less preserved from this obsession, since proofs were easier in that case due to a smoother setting : $L^p$ vector fields instead of singular measures). Notice by the way that the meaning of the word ``cycles'' may be different according to the context: in branched transport, optimal solutions are tree-shaped and two different points are never joined by two different paths, while in congested traffic this is allowed (even encouraged), but what is forbidden is the presence of an oriented cycle, i.e. a completely useless circular path which does not contribute to the divergence but increases total traffic. The recent paper \cite{BraSol} introduces the interesting distinction between ``cycles'' (two different curves connecting the same two points) and ``loops'' (one closed curve).

This correspondence between the Eulerian and the Lagrangian descriptions is crucial to understand in depth the properties of the model, and is also very useful for some proofs. Moreover,  the Eulerian model is also discutable from a modelization point of view, since it allows for cancellations when we want to model two opposite vehicle flows. All these reasons stress the interest for Lagrangian models and, indeed, many papers have been devoted to this topic from a rater abstract point of view: for instance the papers by Paolini-Stepanov, Georgiev-Stepanov and Brasco-Petrache which exactly deal with these representation questions for currents, either in the Euclidean or metric setting, and they do it both for the applications to the transport optimization problems, and for geometric purposes: \cite{PaoliniStepanovFlat,PaoSte2012, PaoSte2013,GeoSte, BraPet}. These papers propose different approaches than Smirnov's one (in particular, Paolini-Stepanov, who prove in \cite{PaoSte2012, PaoSte2013} the same result as Smirnov in the general setting of metric spaces, go the opposite way: first they prove, by approximation from a network setting, that vector fields with no cycles are induced by a $Q$, then they pass to the general case). 

For the congestion case, \cite{BraCarSan} proposed an equivalence between the Lagrangian problem with $Q$ and the Eulerian one with $v$, with an explicit way to construct $Q$ from $v$ following an idea that Dacorogna and Moser developed in \cite{DacMos} for diffeomorphism purposes and first used in transport in \cite{EvaGan}. This required a regularity proof, but in particular an assumption on $\mu$ and $\nu$, which were supposed to have densities bounded from below. This assumption is indeed very important, and explains what is the true obstruction to the equality $v=v_Q$: the problem are not cycles, but cycles turning where there is no mass! indeed, there can be cyclical vector fields $v$ of the form $v_Q$, but in this case both $\mu$ and $\nu$ should give some positive mass to somewhere along the cycle. On the contrary, a cycle located outside of the supports of $\mu$ and $\nu$, cannot for sure be induced by an admissible $Q$ even if its divergence is equal to $\mu-\nu$. Notice also that \cite{BraCar} also proves, by Dacorogna-Moser techniques, the equality of the minimal values of the Eulerian and the Lagrangian problems for the congestion model, but does not take care of the representation of the optimal $v$ as a $v_Q$.

In the present paper we still use Dacorogna-Moser and we give (in Section 1, Proposition \ref{exiQ}) a new proof of Smirnov's Theorem C, which states that every $v$ can be decomposed as the sum of a vector field $v_Q$ induced by a traffic plan and a cycle $v-v_Q$, with no mass loss, i.e. $||v||=||v-v_Q||+||v_Q||$. The proof shares something of the spirit of Smirnov's one, which also followed the integral curves of some vector fields, but is shorter, since it does not require to start the analysis from the cyclical case. Not only, exactly as from Smirnov's result, one can conclude the following fact: for every $v$ there is a vector field $v_Q$ with $|v_Q|\leq |v|$ and $|v_Q|\neq |v|$ unless $v=v_Q$. This allows to prove that optimal vector fields $v$ are of the form $v=v_Q$ every time that we minimize $F(|v|)$ and $F$ has some strict monotonicity properties (Theorem  \ref{car of min}). We will show in particular the application of this result to the case of the Beckmann problem $\min \int \!|v|$.  What is quite astonishing, but it is not really a novelty of this paper (just a different way of stating the results), is the fact that there is no need to care about cycles, or to mention them, even if it happens that the possible difference $v-v_Q$ should indeed come from the presence of cycles in $v$ which do not correspond to some mass in the measures $\mu$ and $\nu$. This was indeed one of the first motivations of the paper (proving $v=v_Q$ for optimal flows $v$ without mentioning cycles), but it turned out later that a new proof of Theorem C by Smirnov was also possible.

No really new result is contained in this paper, just some new - and hopefully simpler - approaches. Part of this framework and of these ideas have been used both for the presentation in the Mini-symposium {\it ``Des probabilités aux EDP par le transport optimal''} at the conference SMAI 2013, and for the mini-course the author gave in August 2013 at MSRI in Berkeley during the program {\it ``Optimal Transport: Geometry and Dynamics''.}

\section{Traffic intensity and traffic flows for measures on curves}

\subsection{Definitions and first properties}

We introduce in this section the main objects of our analysis, i.e. a vector field $v_Q$, called traffic flow, and a scalar intensity $i_Q$ associated to a probability measure $Q$ on the set of paths.  

Let us introduce some notations and be more precise. First, let us denote the space of scalar measures on $\Omega$ by $\M(\Omega)$, the set of positive measures by $\M_+(\Omega)$ and the space of vector measures valued in $\R^d$ by $\M^d(\Omega)$. For a vector measure $v\in\M^d(\Omega)$, we denote by $|v|\in\M_+(\Omega)$ its total variation measure, which is such that $v=\xi\cdot|v|$ with $|\xi|=1,\,|v|-$a.e.. The norm in the space $\M^d(\Omega)$ is given by $||v||=|v|(\Omega)$. Notice that we also have $||v||=\sup \{\int X\cdot dv\,:\,||X||_\infty\leq 1\}$. In all the paper $\Omega$ is considered to be a compact set, and measures on $\Omega$ can give mass to its boundary (we do not make distinctions between $\Omega$ and $\bar\Omega$). In all these measure spaces, the symbol $\deb$ denotes convergence in duality with $C^0$ functions and replaces the symbol $\destar$ for simplicity.

Given an absolutely curve $\omega:[0,1]\to \Omega$ and a continuous function $\varphi$, let us set
\begin{equation}
L_{\varphi}(\sigma):=\int_0^1 \varphi(\omega(t)) \vert \omega'(t)\vert dt.
\end{equation}
This quantity is the length of the curve weighted with the metric $\varphi$. When we take $\varphi=1$ we get the usual length of $\omega$ and we denote it by $L(\omega)$ instead of $L_1(\omega)$.

We consider probability measures $Q$ on the space $C:=\Lip([0,1], \Omega)$. We restrict ourselves to measures $Q$ such that $\int L(\omega)\,dQ(\omega)<+\infty$: these measures will be called {\it traffic plans}, according to a terminology introduced in \cite{BerCasMor}. We endow the space $C$ with the uniform convergence.  Notice that Ascoli-Arzelà's theorem guarantees that the sets $\{\omega\in C\,:\,\Lip(\omega)\leq c\}$ are compact for every $c$.  We will associate two measures on $\Omega$ to such a $Q$. The first is a scalar one, called {\it traffic intensity} and denoted by $i_Q\in \M_+(\Omega)$; it is defined by
\[
\int \varphi \,d i_Q:= \int_{C} \Big(\int_0^1 \varphi(\omega(t)) \vert \omega'(t)\vert dt \Big) dQ(\omega)
=\int_C L_{\varphi}(\omega) d Q(\omega).
\]
for all $\varphi\in C(\Omega, \R_+)$. This definition (taken from \cite{CarJimSan}) is a generalization of the notion of transport density which is nowadays classical in the Monge case of optimal transport theory, see \cite{EvaGan,FelMcC,ambpra,simpleproof}. The interpretation is the following: for a subregion $A$, $i_Q(A)$ represents the total cumulated  traffic in $A$ induced by $Q$, i.e. for every path we compute ``how long'' does it stay in $A$, and then we average on paths. 

We also associate a vector measure $v_Q$ to this traffic plan $Q$, defined  through 
\[\forall X\in C(\Omega; \R^d)\quad\int_{\Omega} X\cdot dv _Q:=\int_{C}  \left(\int_0^1 X(\omega(t))\cdot \omega' (t) dt\right) d Q(\omega).\]
We will call $v_Q$ {\it traffic flow} induced by $Q$. Taking a gradient field $X=\nabla \psi$ in the previous definition yields
\[ \int_{\Omega} \nabla \psi \cdot d { v} _Q=\int_{C}  [\psi(\omega(1))-\psi(\omega(0))] d Q(\omega)=\int_{\Omega} \psi \,d((e_1)_\# Q-(e_0)_\# Q).\]
From now on, we will restrict our attention to {\it admissible} traffic plans $Q$, i.e. traffic plans such that $(e_0)_\# Q=\mu$ and $(e_1)_\# Q=\nu$, where $\mu$ and $\nu$ are two prescribed probability measures on $\Omega$.  This means that
\[ \nabla\cdot v_Q=\mu-\nu \]
and hence $v_Q$ is an admissible flow connecting $\mu$ and $\nu$. Notice that the divergence is always considered in a weak (distributional) sense, and automatically endowed with Neumann boundary conditions, i.e. when we say $\nabla\cdot v=f$ we mean $\int \nabla\psi\cdot dv=-\int \psi \,df$ for all $\psi\in C^1(\Omega)$, without any condition on the boundary behavior of the test function $\psi$.
 
Coming back to $v_Q$, it is easy to check that $\vert v_Q \vert \leq i_Q,$
where $|v_Q|$ is the total variation measure of the vector measure $v_Q$. 
This last inequality is not in general an equality, since the curves of $Q$ could produce some cancellations (imagine a non-negligible amount of curves passing through the same point with opposite directions, so that $v_Q=0$ and $i_Q>0$).

We need some properties of the traffic intensity and traffic flow.
\begin{prop}\label{trois prop of iQ}
Both $v_Q$ and $i_Q$ are invariant under reparametrization (i.e., if $T:C\to C$ is a map such that for every $\omega$ the curve $T(\omega)$ is just a reparametrization in time of $\omega$, then $v_{T_\#Q}=v_Q$ and $i_{T_\#Q}=i_Q$).

For every $Q$, the total mass $ i_Q(\Omega)$ equals the average length of the curves according to $Q$, i.e. $\int_C L(\omega)\,dQ(\omega)=i_Q(\Omega)$. In particular, $v_Q$ and $i_Q$ are finite measures thanks to the definition of traffic plan.

If $Q_n\deb Q$ and $i_{Q_n}\deb i$, then $i\geq i_Q$.

If $Q_n\deb Q$, $v_{Q_n}\deb v$  and $i_{Q_n}\deb i$, then $||v-v_Q||\leq i(\Omega)-i_Q(\Omega)$. In particular, if $Q_n\deb Q$ and $i_{Q_n}\deb i_Q$, then $v_{Q_n}\deb v_Q$.
\end{prop}
\begin{proof} The invariance by reparametrization comes from the fact that both $L_\varphi(\omega) $ and $\int_0^1X(\omega(t))\cdot \omega'(t)dt$ do not change under reparametrization.

The formula $\int_C L(\omega)\,dQ(\omega)=i_Q(\Omega)$ is obtained from the definition of $i_Q$ by testing with the function $1$.

To check the inequality $i\geq i_Q$, fix a positive test function $\varphi\in C(\Omega)$ and suppose $\varphi\geq \ve_0>0$. Write
\begin{equation}\label{vers igeq}
\int\varphi\,di_{Q_n}=\int_{C}  \left(\int_0^1 \varphi(\omega(t))| \omega' (t)| dt\right) d Q_n(\omega).
\end{equation}
Notice that the function $C\ni\omega\mapsto L_\phi(\omega)=\int_0^1 \varphi(\omega(t))|\omega' (t)| dt$ is positive and lower-semi-continuous w.r.t. $\omega$. Indeed, if we take a sequence $\omega_n\to \omega$, from the bound $\varphi\geq \ve_0>0$ we can assume that $\int|\omega'_n (t)| dt$ is bounded. Then we can infer $\omega'_n\deb\omega'$ weakly (as measures, or in $L^1$), which implies, up to subsequences, the existence of a measure $\xi\geq |\omega'|$ such that $|\omega'_n|\deb\xi$; moreover, $\varphi(\omega_n(t))\to \varphi(\omega(t))$ uniformly, which gives $\int \varphi(\omega_n(t))|\omega'_n(t)|dt\to \int \varphi(\omega(t))\xi(t)dt\geq \int \varphi(\omega(t))|\omega'(t)|dt$.

This allows to pass to the limit in \eqref{vers igeq}, thus obtaining
$$
\int\varphi\,d i=\lim_n\int\varphi\,di_{Q_n}=\liminf_n\int_{C} L_\varphi(\omega)d Q_n(\omega)\geq \int_{C}  L_\varphi(\omega) d Q(\omega)=\int\varphi\,d i_Q.
$$
If we take an arbitrary test function $\varphi$ (without a lower bound), add a constant $\ve_0$ and apply the previous reasoning we get $\int (\varphi+\ve_0)di\geq \int (\varphi+\ve_0)di_Q$ and, letting $\ve_0\to0$, since $i$ is a finite measure and $\varphi$ is arbitrary, we get $i\geq i_Q$.

To check the last property, fix a bounded vector test function $X$ and a number $\lambda>1$ and look at 
\begin{multline}\label{addandsub}
\int X\cdot dv_{Q_n}=\int_{C}  \left(\int_0^1 X(\omega(t))\cdot \omega' (t) dt\right) d Q_n(\omega)\\=\int_{C}  \left(\int_0^1 X(\omega(t))\cdot \omega' (t) dt+\lambda||X||_\infty L(\omega)\right) d Q_n(\omega)-\lambda||X||_\infty i_{Q_n}(\Omega),
\end{multline}
where we just added and subtracted the total mass of $i_{Q_n}$, equal to the average of $L(\omega)$ according to $Q_n$.

Now notice that $C\ni\omega\mapsto \int_0^1 X(\omega(t))\cdot \omega' (t) dt+\lambda||X||_\infty L(\omega)\geq (\lambda-1)||X||_\infty L(\omega)$ is l.s.c. in $\omega$ (use the same argument as above, noting that, if we take $\omega_n\to\omega$, we may assume $L(\omega_n)$ to be bounded, and obtain $\omega'_n\deb\omega'$). This means that if we pass to the limit in \eqref{addandsub} we get
\begin{multline*}
 \int\! X\cdot dv=\liminf_n \int\! X\cdot dv_{Q_n}\\\geq \int_{C}\!  \left(\int_0^1 \!\!X(\omega(t))\cdot \omega' (t) dt+\lambda||X||_\infty L(\omega)\!\right) d Q(\omega)-\lambda||X||_\infty i(\Omega)\\=\int\! X\cdot dv_Q+\lambda ||X||_\infty (i_Q(\Omega)-i(\Omega)).
\end{multline*}
By replacing $X$ with $-X$ we get
$$ \left|\int\! X\cdot dv-\int\! X\cdot dv_Q\right|\leq \lambda ||X||_\infty(i(\Omega)-i_Q(\Omega)).$$
Letting $\lambda\to 1$ and taking the supremum over $X$ with $||X||_\infty\leq 1$ we get $||v-v_Q||\leq i(\Omega)-i_Q(\Omega)$.

The very last property is evident, indeed one can assume up to a subsequence that $v_{Q_n}\deb v$ holds for a certain $v$, and $i=i_Q$ implies $v=v_Q$ (which also implies the full convergence of the sequence).
\end{proof}

\subsection{Where Dacorogna-Moser comes into play}
 
Let us start from a particular case of the construction in \cite{DacMos}.\smallskip

\noindent {\bf Construction :} Suppose that $v:\Omega\to\R^d$ is a Lipschitz vector field parallel to the boundary (i.e. $v\cdot n_\Omega=0$ on $\partial\Omega$) with $\nabla\cdot v=f_0-f_1$, where $f_0,f_1$ are positive probability densities which are Lipschitz continuous and bounded from below. Then we can define the non-autonomous vector field $\tilde v(t,x)$ via
$$\tilde v(t,x)=\frac{v(x)}{f_t(x)}\quad\mbox{where }f_t=(1-t)f_0+tf_1$$
and consider the Cauchy problem 
\begin{equation*}
\begin{cases}y_x'(t)=\tilde v(t,y_x(t))\\y_x(0)=x\end{cases},
\end{equation*}
We define the map $Y:\Omega\to C$ through $Y(x)=y_x(\cdot)$, and we look for the measures $Q=Y_\# f_0$ and $\rho_t:=(e_t)_\#Q$. From standard arguments on the link between the flows of the ODE and the continuity equation,  $\rho_t$ solves
$\partial_t\rho_t+\nabla\cdot(\rho_t \tilde v_t)=0$. Yet, it is easy to check that $f_t$ also solves the same equation since $\partial_t f_t=f_1-f_0$ and $\nabla\cdot (\tilde v f_t)=\nabla\cdot w=f_0-f_1$.
By uniqueness arguments (which are valid because $\tilde v_t$ is Lipschitz continuous), from $\rho_0=f_0$ we infer $\rho_t=f_t$. 

In particular, $x\mapsto y_x(1)$ is a transport map from $f_0$ to $f_1$.
\smallskip

It it interesting to compute what are the traffic intensity and the traffic flow associated to the measure $Q$ in Dacorogna-Moser construction. Fix a scalar test function $\varphi$:
\begin{multline*}
\int_{\Omega} \varphi\, di_{Q}=\int_{\Omega} \int_0^1 \varphi(y_x(t)) \vert \tilde v(t,y_x(t))\vert dt f_0(x)dx\\
=\int_0^1 \int_{\Omega} \varphi(y) \vert \tilde v(t,y)\vert f_t(y) dy dt 
=\int_{\Omega}  \varphi(y) \vert v(y) \vert dy
\end{multline*}
so that  $i_{Q}=\vert v\vert$. Analogously, fix a vector test function $X$
\begin{multline*}
\int_{\Omega} X\cdot dv_{Q}=\int_{\Omega} \int_0^1 X(y_x(t)) \cdot \tilde v(t,y_x(t))dt f_0(x)dx\\
=\int_0^1 \int_{\Omega} X(y) \cdot \tilde v(t,y) f_t(y) dy dt 
=\int_{\Omega}  X(y) \cdot v(y)  dy,
\end{multline*}
which shows $v_Q=w$ (indeed, in this case we have $|v_Q|=i_Q$ and this is due to the fact that no cancellation is possible, since all the curves share the same direction at every given point, as a consequence of the uniqueness in Cauchy-Lipschitz theorem).

With these tools we can now get closer to our main result.

\begin{lm}\label{approxveveve}
 Consider two probabilities $\mu,\nu\in\pical(\Omega)$ and a vector measure $v$ satisfying $\nabla\cdot v=\mu-\nu$ in distributional sense (with Neumann boundary conditions). Then, for every domain $\Omega'$ containing $\Omega$ in its interior, there exist a family of vector fields $v^\ve\in C^\infty(\Omega')$ with $v^\ve\cdot n_{\Omega'}=0$, and two families of densities $\mu^\ve,\nu^\ve\in C^\infty(\Omega')$, bounded from below by positive constants $c_\ve>0$, with $\nabla\cdot v^\ve=\mu^\ve-\nu^\ve$ and $\int_{\Omega'} \mu^\ve=\int_{\Omega'}\nu^\ve=1$, weakly converging to $w,\mu$ and $\nu$ as measures, respectively, and satisfying $|v^\ve|\deb |v|$.
\end{lm}
\begin{proof}
First, take convolutions (in the whole space $\R^d$) with a gaussian kernel $\eta_\ve$, so that we get $\hat v^\ve:=v*\eta_\ve$ and $\hat \mu^\ve:=\mu*\eta_\ve$ $\hat \nu^\ve:=\nu*\eta_\ve$, still satisfying $\nabla\cdot \hat v^\ve=\mu^\ve-\nu^\ve$. Since the Gaussian Kernel is strictly positive, we also have strictly positive densities for  $\hat \mu^\ve$ and $\hat \nu^\ve$. These convolved densities and vector field would do the job required by the theorem, but we have to take care of the support (which is not $\Omega'$) and of the boundary behavior.

Let us set $\int_{\Omega'}\hat\mu^\ve=1-a_\ve$ and $\int_{\Omega'}\hat\nu^\ve=1-b_\ve$. It is clear that $a_\ve,b_\ve\to 0$ as $\ve\to 0$. Consider also $\hat v^\ve\cdot n_{\Omega'}$: due to $d(\Omega,\partial\Omega')>0$ and to the fact that $\eta_\ve$ goes uniformly to $0$ locally outside the origin, we also have $|\hat v^\ve\cdot n_{\Omega'}|\leq c_\ve$, with $c_\ve\to 0$.

Consider $u^\ve$ the solution to
$$\begin{cases}\Delta u^\ve=\frac{a_\ve-b_\ve}{|\Omega'|}&\mbox{ inside }\Omega'\\
			\frac{\partial u^\ve}{\partial n}=-\hat v^\ve\cdot n_{\Omega'}&\mbox{ on }\partial\Omega',\\
			\int_{\Omega'}u^\ve=0&\end{cases}$$
and the vector field $\delta^\ve=\nabla u^\ve$. Notice that a solution exists thanks to $\int_{\partial\Omega'}\hat v^\ve\cdot n_{\Omega'}=a_\ve-b_\ve$. Notice also that an integration by parts shows
$$\int_{\Omega'}|\nabla u^\ve|^2=-\int_{\partial\Omega'}u^\ve(\hat v^\ve\cdot n_{\Omega'})-
\int_{\Omega'}u^\ve\left(\frac{a_\ve-b_\ve}{|\Omega'|}\right)\leq C||\nabla u^\ve||_{L^2}(c_\ve+a_\ve+b_\ve),$$
and provides $||\nabla u^\ve||_{L^2}\leq C(a_\ve+b_\ve+c_\ve)\to 0$. This shows $||\delta^\ve||_{L^2}\to 0$.

Now take 
$$\mu^\ve=\hat\mu^\ve\res\Omega'+\frac{a_\ve}{|\Omega'|};\quad \nu^\ve=\hat\nu^\ve\res\Omega'+\frac{b_\ve}{|\Omega'|};\quad v^\ve=\hat v^\ve\res\Omega'+\delta^\ve,$$
and check that all the requirements are satisfied. In particular, the last one is satisfied since $||\delta^\ve||_{L^1}\to 0$ and $|\hat v^\ve|\deb |v|$ by general properties of the convolutions.
\end{proof}
\begin{rem}\label{approxveveverem}
Notice that considering explicitly the dependence on $\Omega'$ it is also possible to obtain the same statement with a sequence of domains $\Omega'_\ve$ converging to $\Omega$ (for instance in the Hausdorff topology). It is just necessary to choose them so that, setting $t_\ve:=d(\Omega,\partial\Omega'_\ve)$, we have $||\eta_\ve||_{L^\infty(B(0,t_\ve)^c)}\to 0$. For the Gaussian kernel, this is satisfied whenever $t_\ve^2/\ve\to\infty$ and can be guaranteed by taking $t_\ve=\ve^{1/3}$.
\end{rem}

With these tools we can now prove
\begin{prop}\label{exiQ}
For every finite vector measure $v\in\M^d(\Omega)$ and $\mu,\nu\in\pical(\Omega)$ with $\nabla\cdot v=\mu-\nu$ there exists a traffic plan $Q\in\pical(C)$ with $(e_0)_\#Q=\mu$ and $(e_1)_\#Q=\nu$  such that $|v_Q|\leq i_Q\leq |v|$, and $||v-v_Q||+||v_Q||=||v-v_Q||+i_Q(\Omega)=||v||$. In particular we have  $|v_Q|\neq |v|$ unless $v_Q=v$.
\end{prop}

\begin{proof}
By means of Lemma \ref{approxveveve} and Remark \ref{approxveveverem} we can produce an approximating sequence $(v^\ve,\mu^\ve,\nu^\ve)\deb(w,\mu,\nu)$ of $C^\infty$ functions supported on domains $\Omega_\ve$ converging to $\Omega$. We apply Dacorogna-Moser's construction to this sequence of vector fields, thus obtaining a sequence of measures $Q_\ve$. We can consider these measures as probability measures on $\Lip([0,1];\Omega')$ (where $\Omega\subset\Omega_\ve\subset\Omega'$) which are, each, concentrated on curves valued in $\Omega_\ve$. They satisfy $i_{Q_\ve}=|v^\ve|$ and $v_{Q_\ve}=v^\ve$. We can reparametrize (without changing their names) by constant speed the curves on which $Q_\ve$ is supported, without changing traffic intensities and traffic flows. This means using curves $\omega$ such that $L(\omega)=\Lip(\omega)$. The equalities
$$\int_C \Lip(\omega)\,dQ_\ve(\omega)=\int_C L(\omega)\,dQ_\ve(\omega)=\int_{\Omega'}i_{Q_\ve}=\int_{\Omega'}|v^\ve|\to |v|(\Omega')=|v|(\Omega)$$
show that $\int_C \Lip(\omega)\,dQ_\ve(\omega)$ is bounded and $Q_\ve$ is tight. Hence, up to subsequences, we can assume $Q_\ve\deb Q$. The measure $Q$ is obviously concentrated on curves valued in $\Omega$. The measures $Q_\ve$ were constructed so that $(e_0)_\#Q_\ve=\mu^\ve$ and $(e_1)_\#Q_\ve=\nu^\ve$, which implies, at the limit, $(e_0)_\#Q=\mu$ and $(e_1)_\#Q=\nu$. Moreover, thanks to Proposition \ref{trois prop of iQ}, since $i_{Q_\ve}=|w_\ve|\deb |v|$ and $v_{Q_\ve}\deb v$, we get $|v|\geq i_Q\geq |v_Q|$ and $||v-v_Q||\leq |v|(\Omega)-i_Q(\Omega)$. This gives $||v-v_Q||+||v_Q||\leq ||v-v_Q||+i_Q(\Omega)\leq||v||$ and the opposite inequality $||v||\leq ||v-v_Q||+||v_Q||$ is always satisfied. 
\end{proof}

\begin{rem}
The previous statement contains that of Theorem C in \cite{smirnov}, i.e. the decomposition of any $v$ into a cycle $v-v_Q$ and a flow $v_Q$ induced by a measure on paths , with $||v||=||v-v_Q||+||v_Q||$.
\end{rem}

\begin{rem}
It is possible to guess what happens to a cycle through this construction. Imagine the following example: $\Omega$ is composed of two parts $\Omega^+$ and $\Omega^-$, with $\spt\mu\cup\spt\nu\subset\Omega^+$ and a cycle of $v$ is contained in $\Omega^-$. When we build the approximations $\mu^\ve$ and $\nu^\ve$ they will give a positive mass to $\Omega^-$, but very small. Because of the denominator in the definition of $\tilde v$, the curves produced by Dacorogna-Moser will follow this cycle very fast, passing many times on each point of the cycle. Hence, the flow $v^\ve$ in $\Omega^-$ is obtained from a very small mass which passes many times. This implies that the measure $Q_\ve$ of this set of curves will disappear at the limit $\ve\to 0$. Hence, the measure $Q$ will be concentrated on curves staying in $\Omega^+$ and $v_Q=0$ on $\Omega^-$. In this way, we got rid of that particular cycle, but the same does not happen for cycles located in regions with positive masses of $\mu$ and $\nu$. In particular nothing guarantees that $v_Q$ has no cycles.
\end{rem}

\section{Application to optimal flows}

In this section we try to get some (easy) conclusions from the previous proofs. 

\begin{teo}\label{car of min}
Consider a minimization problem
$$\min \quad \left\{F(|v|)\;:\; v\in\M^d(\Omega),\,\nabla\cdot v =\mu-\nu\right\}$$
where $F:\M_+(\Omega)\to\R\cup\{+\infty\}$ is an increasing functional, i.e. $\alpha\leq \beta\impl F(\alpha)\leq F(\beta)$.
Then if a minimizer $v$ exists, there is also a minimizer of the form $v_Q$, where $Q$ is an admissible traffic plan connecting $\mu$ to $\nu$.
Moreover, if $F$ is strictly increasing (i.e. $\alpha\leq \beta, F(\alpha)=F(\beta)\impl \alpha=\beta$), any minimizer $v$ must be of the form $v_Q$.
\end{teo}
\begin{proof}
These facts are easy consequences of Proposition \ref{exiQ}. Indeed, take a minimizer $v$ and build a traffic plan $Q$ such that $|v_Q|\leq |v|$. By the monotonicity of $F$ it is clear that $v_Q$ must also be a minimizer. Moreover, it must satisfy $F(|v|)=F(|v_Q|)$, since the value  $F(|v_Q|)$ cannot be smaller than the minimal value $F(|v|)$. Hence, if $F$ is strictly increasing, we must have $|v|=|v_Q$ but in this case Proposition \ref{exiQ} allows to conclude $v=v_Q$.
\end{proof}

 In order to complete our analysis, we use the previous result to characterize the optimal flows in the Beckmann's version of the optimal transport problem. To do that, let us first recall the main points of Beckmann's model.
 
 The problem, that Beckmann proposed in \cite{beck} under the name {\it Continuous model of transportation} is the following: given two measures $\mu$ and $\nu$, find the vector field $v$ satisfying $\nabla\cdot v=\mu-\nu$ with minimal $L^1$ norm. In the language of measures, this becomes 
 $$(PB)\qquad\min\quad\left\{ |v|(\Omega)\;:\; v\in\M^d(\Omega),\,\nabla\cdot v =\mu-\nu\right\}.$$
The existence of an optimal measure $v$ is straightforward, and in some cases one can also prove that it is actually absolutely continuous, with $L^1$ or $L^p$ density (see \cite{simpleproof}  for the most recent results on this issues). 

It happens that, even if Beckmann was not aware of it, this problem is indeed strongly connected with the Monge-Kantorovich problem for cost $c(x,y)=|x-y|$. Indeed, the minimal value of (PB) is equal to the minimal value of 
$$(PK)\qquad\min\left\{\int_{\Omega\times\Omega}|x-y|\,\,d\gamma\;:\;\gamma\in\Pi(\mu,\nu)\right\},$$
where the set
$\Pi(\mu,\nu)$ is the set of the so-called {\it transport plans}, i.e. $\Pi(\mu,\nu)=\{\gamma\in\pical(\Omega\times\Omega):\,(\pi_x)_{\#}\gamma=\mu,\,(\pi_y)_{\#}\gamma=\nu,\}$
where $\pi_x$ and $\pi_y$ are the two projections of $\Omega\times\Omega$ onto $\Omega$. Moreover, it is possible to produce a minimizer $v_{[\gamma]}$ for (PB) starting from a minimizer $\gamma$ for (PK) by defining $v_{[\gamma]}$ as a measure in the following way:
$$
\int X\,dv_{[\gamma]}:=\int_{\Omega\times\Omega}\int_0^1 \omega'_{x,y}(t)\cdot X(\omega_{x,y}(t))dt\, d\gamma,
$$
for every $X\in C^0(\Omega;\R^d)$, $\omega_{x,y}$ being a parametrization of the segment $[x,y]$. It is clear that this definition is nothing but a particular case of the definition of $v_Q$, when we take $Q=S_\#\gamma$ and $S$ is the map $\Omega\times\Omega\ni (x,y)\mapsto \omega_{x,y}\in C$ (where, to avoid ambiguities, we choose the segment $\omega_{x,y}$ parametrized with constant speed: $\omega_{x,y}(t)=(1-t)x+ty$).

The final result we present is exactly the fact that all the minimizers $v$ of (PB) must be of the form $v_{[\gamma]}$.

\begin{teo}\label{vvQvgam}
Let $v$ be optimal in $(PB)$: then there is an optimal transport plan $\gamma$ such that $v=v_{[\gamma]}$.
\end{teo}
\begin{proof} Thanks to Theorem \ref{car of min}, since $F(\alpha)=|\alpha|(\Omega)$ is strictly increasing, we can find an admissible traffic plan $Q\in \pical(C)$ with $(e_0)_\#Q=\mu$ and  $(e_1)_\#Q=\nu$ such that $v=v_Q$. We can assume $Q$ to be concentrated on curves parametrized by constant speed. 
The statement is proven if we can prove that $Q=S_\#\gamma$ with $\gamma$ an optimal transport plan.

Indeed, using again the optimality of $v$ and Proposition \ref{exiQ}, we get
\begin{multline*}
\min(PB)=||v||=i_Q(\Omega)= \int_C L(\omega)\,dQ(\omega)\geq \int_C|\omega(0)-\omega(1)|dQ(\omega)\\=\int_{\Omega\times\Omega}|x-y|\,d((e_0,e_1)_\# Q)\geq \min (PK).$$
\end{multline*}
The equality $\min(PB)=\min(PK)$ implies that all these inequalities are equalities. In particular $Q$ must be concentrated on curves such that $L(\omega)=|\omega(0)-\omega(1)|$, i.e. segments. Also, the measure $(e_0,e_1)_\# Q$, which belongs to $\Pi(\mu,\nu)$, must be optimal in $(PK)$. This concludes the proof.
\end{proof}

\end{document}